\documentclass{amsart}
\usepackage{amssymb,latexsym,amsmath,amsfonts,amsthm,amscd,graphicx,url,color,stmaryrd,mathtools,tikz-cd,hyperref}
\usepackage{diagrams}

\theoremstyle{plain}
\newtheorem{thm}{Theorem}[section]
\newtheorem{prop}[thm]{Proposition}
\newtheorem{cor}[thm]{Corollary}
\newtheorem{lemma}[thm]{Lemma}

\theoremstyle{definition}
\newtheorem{definition}[thm]{Definition}
\newtheorem{remark}[thm]{Remark}

\newtheorem{example}[thm]{Example}

\numberwithin{equation}{section}

\newcommand{\Hom}{\mathrm{Hom}}

\renewcommand{\P}{\mathbb{P}}
\newcommand{\A}{\mathbb{A}}
\newcommand{\Z}{\mathbb{Z}}
\newcommand{\N}{\mathbb{N}}

\newcommand{\Proj}{\mathrm{Proj}}
\newcommand{\spec}{{\rm Spec}}

\newcommand{\Span}{\mathrm{Span}}

\newcommand{\codim}{\mathrm{codim}}

\newcommand{\initial}{\mathrm{in}}
\newcommand{\mm}{\mathfrak{m}}

\newcommand{\bfa}{\mathbf{a}}
\newcommand{\bfb}{\mathbf{b}}
\newcommand{\bfc}{\mathbf{c}}
\newcommand{\bfd}{\mathbf{d}}
\newcommand{\bfe}{\mathbf{e}}
\newcommand{\bfv}{\mathbf{v}}

\def\z{{\zeta}}
\def\a{{\alpha}}
\def\bfd{{\bf d}}
\def\ZZ{{\Bbb Z}}
\def\PP{{\P}}
\def\S{{\Sigma}}
\def\cO{{\mathcal O}}

\DeclareMathOperator{\Sym}{\rm Sym}

\title{Correspondence scrolls}
\author{David Eisenbud \and Alessio Sammartano}

\address[David Eisenbud]{Mathematical Sciences Research Institute, 17 Gauss Way, Berkeley, CA 94720, USA}
\email{de@msri.org}

\address[Alessio Sammartano]{Dpartment of Mathematics, University of Notre Dame, 255 Hurley, Notre Dame, IN 46556, USA}
\email{asammart@nd.edu}

\subjclass[2010]{Primary: 14J40. Secondary: 13H10; 13C40; 13P10;	14J26; 14J28; 14J32; 14M05; 14M12; 14M20}
\keywords{Rational normal scroll; Veronese embedding; join variety; multiprojective space; variety of complexes; variety of minimal degree; double structure; K3  surface; Calabi-Yau scheme; Gorenstein ring; Gr\"obner basis.} 

\begin{document}
\begin{abstract}
This paper  initiates the study of a class of schemes that we call
\emph{correspondence scrolls}, which includes the rational normal scrolls and linearly embedded projective bundle of decomposable bundles, 
as well as degenerate K3 surfaces, 
Calabi-Yau 3-folds, and many other examples.
\end{abstract}

\maketitle 

\section{Introduction}
We will define and study a class of schemes that we call
\emph{correspondence scrolls}. The origin of our interest was in a paper by Frank Schreyer and the first author on the equations and syzygies of degenerate K3 surfaces such as K3 carpets~\cite{EiSc}. Correspondence scrolls are a natural generalization of rational normal scrolls and K3 carpets that includes families of (degenerate) Calabi-Yau 3-folds and many other
examples.

We will define a correspondence scroll $C(Z;\bfb)$ for any  subscheme $Z\subseteq \P^\bfa := \P^{a_1}\times \cdots\times \P^{a_n}$, and any $n$-tuple of non-negative integers
$\bfb = (b_1,\dots,b_n)$. In the special case where $Z$ is reduced, we may define $C(Z;\bfb)$ as follows: 
embed  $\P^{a_i}$ by the $b_i$-th Veronese embedding $\nu_{b_i}:\P^{a_i}\to \P^{{a_i+b_i\choose a_i}-1}$ into a general linear subspace of 
$\P^N$, where $N = \sum_i {a_i+b_i\choose a_i}-1$, and set
$$
C(Z;\bfb) = \bigcup_{(p_1,\dots,p_n)\in Z} \overline{\nu_{b_1}(p_1),\dots, \nu_{b_n}(p_n)}.
$$
Thus $C(Z;\bfb)$ is a union of $(n-1)$-planes. 

In this paper we determine the dimension, degree and multi-graded Hilbert function of
a scheme of the form $C(Z;\bfb)$, and for which $Z$ they are nonsingular for every $\bfb$.
We explain the primary decompositions and Gr\"obner bases of their defining ideals,
 and we determine which ones are Cohen-Macaulay,  Gorenstein, or  numerically Calabi-Yau.
 We give numerous examples, including some new (as far as we know) examples of degenerate Calabi-Yau 3-folds.
 
Recall that rational normal scrolls are the varieties of minimal degree in $\P^{N}$ that contain linear spaces of codimension 1 (the only other varieties of minimal degree are the cones over the Veronese surface in $\P^{5}$, see~\cite{EiHa} for an expository account). 
Perhaps because of their extremal properties, they appear in many contexts in algebraic geometry, for example as ambient spaces of Castlenuovo curves (see for example \cite{Ha}) and canonical curves (see for example \cite{Sch}) and as images of canonical maps of certain varieties of general type (see for example \cite{GaPu1}).

Gallego and Purnaprajna prove in \cite{GaPu} that, on each 2-dimensional rational normal scroll, there is a unique double structure of a scheme that is a degenerate K3 surface in a natural sense. They called these schemes K3 carpets. Some of the interest in these schemes 
comes because the hypersurface section of a K3 carpet is a canonical curve of controlled genus and Clifford index, and this is the point of view taken in \cite{EiSc}.

In the study of K3 carpets in~\cite{EiSc} the authors mention that the equations of K3 carpets can be described as equations of varieties of complexes, coming from certain correspondences of type $(2,2)$ in $\P^{1}\times\P^{1}$. In this paper we  generalize the construction, and show that the resulting ``correspondence scrolls''  have algebraic properties that are frequently easy to analyze.

Here is the general definition:

\begin{definition}[Correspondence scroll]\label{DefinitionCorrespondenceScroll}
Given  a vector $\bfa= (a_1, \ldots, a_n) \in \N_+^n$ and a field $\Bbbk$, consider the polynomial ring 
$$A = \Bbbk\big[x_{i,j} \, :\, 1\leq i \leq n, \, 0 \leq j \leq a_i\big]$$
equipped with the standard $\Z^n$-grading $\deg(x_{i,j}) = \bfe_i \in \N^n$.
The ring $A$ is  the Cox ring of $\P^\bfa := \prod_{i=1}^n\P^{a_i}$ as well as the coordinate ring of 
$ \prod_{i=1}^n\A^{a_i+1}$.
Let $Z $ be a subscheme of $ \prod_{i=1}^n\A^{a_i+1}$ defined by a multigraded ideal $I \subseteq A$. 
Let $\bfb \in \N_+^n$ be another vector and $N = \sum {a_i +b_i \choose a_i} -1$.
We define the \emph{correspondence scroll} $C(Z; \bfb)\subseteq \P^N$ to be the scheme defined by the kernel of the map
$$
S =\Bbbk[z_{i,\alpha}] \to A/I\,:\quad z_{i,\alpha} \mapsto x_i^\alpha,\quad |\alpha| = b_i
$$
where $z_{i,\alpha}$ are variables of degree 1, 
$x_i^\alpha$ denotes a monomial of degree $b_i$ that is the product $x_{i,0}^{\alpha_0}\cdots x_{i, a_i}^{\alpha_{a_i}}$ 
and the indices $\alpha$ on $z_{i,\alpha}$ have weight $b_i$.
\end{definition}

In many cases of interest, $Z$ comes from a closed subscheme of $\P^\bfa$;
 that is, $I$ has no primary component whose radical contains one of the ideals $(x_{i,0},\dots,x_{i,a_i})$.
 Such a subscheme $Z$ is called a \emph{correspondence}.
The scheme $C(Z; \bfb)$ is then set-theoretically the union of the projective $(n-1)$-planes joining the points $p_1,\dots,p_n$
that are ``in correspondence'', in the sense that $(p_1,\dots,p_n) \in Z$, whence
the name we have given the construction.

\begin{example}\label{first examples}
If $I = 0$ then $C(Z; \bfb)\subseteq \P^N$ is  the join variety of
the $b_i$-th Veronese embeddings of $\P^{a_i}$, for $i=1,\dots, n$.

 For a less trivial example we take $a_1 = \cdots =a_n =1$ and take $Z$ to  be the ``small diagonal"
$$
\Delta := \Big\{(p_1,\dots, p_n) \in \prod_n \P^1 \,\big|\, p_1 = \cdots =p_n\Big\}\,.
$$
 In this case $C(Z; \bfb)$ is the rational normal scroll of type $b_1,\dots,b_n$, which we denote by $\S(b_{1}, \dots, b_{n})$.
\end{example}

 We say that a projective scheme $X$ is Calabi-Yau if $\cO_X$ has no intermediate cohomology  and $\omega_X \cong \cO_X$. 
If $X$ is 2-dimensional we say that it is K3. 
These definitions reduce to the usual definitions when the scheme is smooth. 
We will be interested in embeddings of these schemes where they are arithmetically Cohen-Macaulay as well, and then they can be described as those schemes whose homogeneous coordinate rings are Gorenstein of $a$-invariant 0.

 \begin{example}[K3 surfaces]\label{K3 Carpet example} 
For each scroll surface  $C(\Delta;\bfb)=\S(b_1,b_2)\subseteq \P^{b_1+b_2+1}$ there exists a unique K3 double structure supported on it, called a K3 carpet $K(b_{1}, b_{2})$, as shown in \cite{GaPu}. 
The scheme $K(b_{1}, b_{2})$ is arithmetically Cohen-Macaulay,  has degree twice the degree of the scroll $\S(b_{1}, b_{2})$, 
and the hyperplane section of $K(b_{1}, b_{2})$ is a canonically embedded rational ribbon of genus $b_{1}+b_{2}+1$ and Clifford index $\min(b_{1}, b_{2})$.

As described by Eisenbud and Schreyer in \cite{EiSc}, the K3 carpet $K(b_{1}, b_{2})$ is the correspondence scroll $C(2\Delta;\bfb)$, where $\Delta \subseteq \P^1\times \P^1$ denotes the  diagonal as above.

Generalizing the properties given for these examples in \cite{EiSc} was the original motivation for this paper. 
For instance, $K(b_{1}, b_{2})$ is  numerically K3 by Theorem \ref{TheoremGorenstein} and it
has degree $2(b_1+b_2)$ by Corollary \ref{CorollaryDegreeDivisor}.
\end{example}

 \begin{example}[Calabi-Yau threefolds]\label{CY3} 
 As a first extension of the theory above, we note that
Corollary \ref{CorollaryCompleteIntersection} yields examples of  Calabi-Yau threefolds.
For instance,
if $Z$ is a divisor of type $(3,2)$ in $\P^2 \times \P^1$, 
then  $C(Z;\bfb)\subseteq \P^{{b_1+2 \choose 2} + b_2}$ is a Calabi-Yau threefold for every $\bfb=(b_1,b_2)$, and 
it has degree $3b_1b_2+2b_1^2$ by Corollary \ref{CorollaryDegreeDivisor}.
\end{example}

\begin{example}[Schemes with irrelevant components]\label{CY2}
Another family of examples of Calabi-Yau threefolds is given by taking $Z$ to be
a complete intersection of two trilinear hypersurfaces in
$\A^2\times\A^2\times\A^2$. 
Then  $C(Z;\bfb)\subseteq \P^{b_1+b_2+b_3+2}$ is numerically a Calabi-Yau threefold for every $\bfb$. The scheme $Z$ cannot be considered as a subscheme of 
$\P^1\times\P^1\times\P^1$, since each of the 3 ``irrelevant'' ideals is necessarily a component. In this case $C(Z;\bfb)$ has degree 
$b_2b_3+b_1b_3+b_1b_2+2(b_1+b_2+b_3)$. See Example~\ref{CY2a} for an explanation.
\end{example}

\section{Defining ideal, Dimension and Degree}

The ideal of $C(Z;\bfb)\subseteq \P^N$ is easy to describe explicitly. 
First, the ideal of the join variety $C(\P^\bfa; \bfb)$ is just the sum of the ideals $I_1, \dots, I_n$ of the different Veronese varieties $\nu_{b_i}(\P^{a_i})\subseteq \P^N$. These may be expressed in well-known ways as ideals of $2$-minors of matrices of linear forms---see for example \cite{Ei}. 
The ideal of $C(Z;\bfb)$ is thus of the form
$J+\sum_{i=1}^n I_i$, where $J$ is derived from the ideal of $Z$ as follows:

\begin{prop}\label{generators}
 With notation above, suppose that the ideal $I$ of $Z\subset \prod_{i=1}^n\A^{a_i+1}$ is generated by multigraded forms
 $g_1, \ldots, g_s$, with $\deg(g_j)=(d_{j,1}, \dots, d_{j,n})$. 

 Let  $J$ be the ideal of $S$ generated by the pullback of the forms in
\def\bfe{{\bf e}}
$
g_j\cdot A_{\bfc_j},
$
where $\bfc_j = (c_{j,1}, \ldots, c_{j,n})$  and the $c_{j,i}$ are chosen so that $d_{j,i}+c_{j,i}$ is the smallest
multiple of $b_i$ that is $\geq d_{j,i}$.

The ideal of $C(Z; \bfb)$  is equal to $J+ \sum_{i=1}^n I_i$.
\end{prop}
\begin{proof}
The homogeneous coordinate ring of the join variety, which is defined by the ideal
$\sum_{i=1}^n I_i$,  may be identified with the subring of $A$ generated by the elements of
multidegrees $(r_1,\dots, r_n)$ such that each $r_i$ is a multiple of $b_i$. With this identification, the ideal of $C(Z;\bfb)$ in the join variety is the intersection of the ideal of $Z$ with the subring
generated by these elements, from which the conclusion follows.
\end{proof}

Recall that the  Chow ring  of $\P^\bfa = \P^{a_1}\times \cdots\times \P^{a_n}$  
is 
$$
\ZZ[\z_{1}, \dots, \z_{n}]/(\z_{1}^{1+a_{1}},\dots,\z_{n}^{1+a_{n}}).
$$

\begin{thm}\label{dim and deg}
 Let $I\subset A$ be a multigraded ideal and let $Z$ be the corresponding subscheme of $\prod_{i=1}^n\A^{a_i+1}$. 

\begin{enumerate}
\item  The dimension of 
 $C(Z;\bfb) \subseteq \P^N$ is one less than the dimension of 
 $Z\subset \prod_{i=1}^n\A^{a_i+1}.$ 

\item If the multigraded Hilbert function of $A/I$ is
$
H_{A/I}(t_1,\dots, t_n)
$
then  the Hilbert function of $C(Z;\bfb)$ is
$$
H_{C(Z;\bfb)}(s) = \sum_{\sum t_i = s } H_{A/I}(b_1t_1,\dots, b_nt_n).
$$
\item If $Z\subset \P^\bfa$ is a subscheme 
of dimension $d$ and class 
$$
c(Z) = \sum_{\a \in \ZZ^{n}} r_{\a}\z^{\bfa- \a} \in {\rm Chow}(\PP^{\bfa})
$$
then  the degree of $C(Z;\bfb)$ is
$$
\sum_{|\a| = d}\ r_\alpha  \prod_{i = 1}^{n}b_{i}^{\a_{i}}.
$$
\end{enumerate}
\end{thm}

\begin{proof} 
(1) The coordinate ring of $Z\subset \prod_{i=1}^n\A^{a_i+1}$ and the homogeneous coordinate ring of  $C(Z;\bfb) \subseteq \P^N$ have the same Krull dimension, 
because the former is an integral extension of the latter. 

\noindent(2)
The multigraded Hilbert function of $C(Z;\bfb)$ is 
$$
H_{C(Z;\bfb)} (t_1,\dots,t_n) = H_{A/I}(b_1t_1,\dots, b_nt_n),
$$
and the formula follows.

\noindent (3) The degree is a linear function of the top dimensional part of the Chern class of $Z$,
so it suffices to do the case where the class of $Z$ is the monomial $\zeta^{\bfa -\a}$, and we may assume
$Z = \prod \PP^{\a_i}\subseteq \prod \PP^{a_i}$. Since the restriction of the $b_i$-th Veronese map 
$\nu_{b_i}: \P^{a_i}\to \P^{{a_i+b_i\choose a_i}-1}$ to a linear subspace is again of the form
$\nu_{b_i}$, we see that $C(Z;\bfb)$ is the join of the $b_i$-th Veronese embeddings of the $\P^{\a_i}$.

The degree of the $b_{i}$-th Veronese embedding  $\nu_{b_i}(\P^{\a_{i}})$ is $b_{i}^{\a_{i}}$. 
Thus a general plane in $\P^N$ of codimension equal to $\dim C(Z;\bfb) = \sum_i a_i$ meets the linear span
of $\nu_{b_i}(\P^{\a_i})$ in $b_{i}^{\a_{i}}$ points, and thus meets $C(Z;\bfb)$ in the disjoint union of the  $(n-1$)-planes 
spanned by one point from each of the  intersections with the $\nu_{b_i}(\P^{\a_i})$, a total
of $\prod_{i=1}^n b_i^{\a_i}$ planes, which has degree $\prod_{i=1}^n b_i^{\a_i}$ as required.
\end{proof}

For example,  Theorem \ref{dim and deg} gives the well-known degree of the  the rational normal scroll $\S(b_{1}, \dots, b_{n})\subseteq \PP^{N}$ as $\sum_{i=1}^n b_{i}$.

We single out the interesting case of a divisor in $\P^\bfa$.

\begin{cor}\label{CorollaryDegreeDivisor}
Let  $Z\subset \P^\bfa$ be a divisor of type $(d_1, \ldots, d_n)$, 
then
$$
\deg C(Z;\bfb)=\sum_{i=1}^n d_i b_i^{a_i-1}\prod_{j\neq i} b_j^{a_j}.
$$
\end{cor}

\def\bfe{{\bf e}}
\begin{proof}
 The class of $Z$ is $\sum d_i\z_i = \sum_i d_i\zeta^{\bfa - (\bfa -\bfe_i)}$
 where $\bfe_i$ denotes the unit vector with a 1 in the $i$-th position and zeros elsewhere. 
\end{proof}
  
In the more general case where the ideal of $Z$ has irrelevant components we do not have such a simple formula for the degree.
Consider a multigraded ideal $I\subset A$, and $Z\subset \prod_{i=1}^n\A^{a_i+1}$ the corresponding affine scheme.
If $Z = \cup_{i}Z_{i}$ is a primary decomposition of $Z$, then $C(Z; \bfb)$ has a primary decomposition
$\cup_i C(Z_{i};\bfb)$ and the degree is the sum of the degrees of the primary components of maximal dimension, so it suffices to treat the primary case.

For example, suppose that $Q = Q'+Q''$, where  $Q'\subseteq(x_{1 ,0},\dots, x_{m,a_{m}})$ is a relevant primary ideal and $Q''$ is $(x_{m+1,0},\dots, x_{n,a_{n}})-$primary.
 Let
 $$
 A' =  \Bbbk[x_{1,0},\dots,x_{m,a_{m}}],\ \A' = \spec\,A',
 $$
 and let $Z'\subseteq \A'$ be the scheme defined by 
 $Q'\cap A'$.
Suppose first that
$b_{m+1}, \dots, b_{n}$ are sufficiently large so that $Q'' \supset \mm_{j}^{b_{j}}$ for
  $j= m+1,\dots, n$. In this case, the scheme $C(Z;\bfb)$ is contained in $\A'\subseteq \prod_{i=1}^n\A^{a_i+1}$, and coincides with $C(Z';(b_1,\dots, b_{m}))$. 
  In particular, the degree of $C(Z;\bfb)$ is equal to that of $C(Z';(b_1,\dots, b_{m}))$.

\begin{example}[Example~\ref{CY2} continued]\label{CY2a}
The ideal $I \subseteq \Bbbk[x_{1,0},\dots, x_{3,1}]$ generated by a regular sequence of two trilinear forms on
$\A^2\times\A^2\times\A^2$
has the three irrelevant ideals
$$
(x_{1,0},x_{1,1}), \, (x_{2,0},x_{2,1}), \, (x_{3,0},x_{3,1})
$$
among its primary components. If $Z_1 = V(x_{1,0},x_{1,1})$ then by the discussion
above $C(Z_1;(b_1,b_2,b_3))$ is the cone over the join of $\nu_{b_2}(\P^1)$ and $\nu_{b_3}(\P^1)$,  and thus of degree $b_2b_3$; likewise for the other irrelevant components. 
Furthermore, each trilinear form represents the class
$\z_1+\z_2+\z_3 \in {\rm Chow}(\P^1\times\P^1\times\P^1)$, so the ``relevant" part of the intersection is $(\z_1+\z_2+\z_3)^2 = 2(\z_1\z_2+\z_1\z_3+\z_2\z_3)$, which is twice the class of the diagonal embedding of $\P^1$, and for a subscheme of this class, Theorem~\ref{dim and deg}
implies that $\deg C(Z';(b_1,b_2,b_3)) = 2(b_1+b_1+b_3)$. 
Putting these together we get
$$
\deg C(Z;(b_1,b_2,b_3)) = (b_2b_3+b_1b_2 + b_2b_3) +2 (b_1+b_2+b_3).
$$
By Theorem~\ref{TheoremGorenstein} this is an arithmetically Cohen-Macaulay Calabi-Yau 3-fold in $\P^{b_1+b_2+b_3+2}$.
\end{example}

It would be interesting to investigate whether the Calabi-Yau threefolds of Example~\ref{CY2} are smoothable or smooth points in the Hilbert scheme, in analogy to the case of canonical ribbons and the K3 surfaces of Examples~\ref{K3 Carpet example} and \ref{K3 example}. 
See \cite{BaEi} and \cite{GaPu}.

For smaller values of $\bfb$, the scheme $C(Z;\bfb)$ is contained in an infinitesimal neighborhood of 
$\A'$, and has degree equal to  product of the degree of $C(Z';\bfb)$ and the length of the subring of 
$\Bbbk[x_{m+1,0},\dots, x_{n,a_{n}}]/Q''$ 
generated by the generators of the ideal $\sum_{i> m} \mm_{i}^{b_{i}}$.

\begin{example}
Suppose $n=2, a_1 = a_2 = 1$, and 
$$
I = (x_{1,0}^m,x_{1,1},x_{2,0})\subseteq \Bbbk[x_{1,0},x_{1,1},x_{2,0},x_{2,1}]
$$ 
for some $m\in \N$. 
If $Z\subseteq \A^2\times \A^2$ is the scheme corresponding to $I$ then 
$C(Z;(b_1,b_2))$ is the subscheme of $\P^N$ consisting of a
non-reduced point of degree $\lceil m/b_1\rceil$.  Indeed, the ideal of 
$C(Z;(b_1,b_2))\subseteq \P^{N}$  is 
$$
\Big(\big\{z_{1,\alpha} \mid \alpha\neq (b_1,0)\big\}\Big)
+\Big(z_{1,(b_1,0)}^{\lceil m/b_{1}\rceil}\Big)
+\Big(\big\{z_{2,\alpha} \mid \alpha\neq (0,b_{2})\big\}\Big). 
$$
Thus the degree formula of Theorem~\ref{dim and deg} holds whenever $b_1 \geq m$.
\end{example}

\section{Alternate Representations: Images of vector bundles and Varieties of complexes}\label{SectionBundles}

One of the standard descriptions of a rational normal scroll $\S(\bfb)$ is as the image of the projectivized
vector bundle $\PP(\oplus_{i}\cO_{\P^{1}}(b_{i}))$ under the linear series $\cO_{\P}(1)$. A similar description
is valid whenever $I$ has no irrelevant components, so that $Z$ may be considered as a projective scheme.

For each $i = 1, \ldots, n$ we consider the $b_i$-uple embedding
$$
\nu_{b_i} : \P^{a_i} \hookrightarrow \P^{{a_i +b_i\choose a_i}-1} \subseteq \P^N.
$$ 
Consider the incidence correspondence in $\PP^{\bfa} \times \PP^{N}$ given by
$$
\Gamma := \left\{(p, x)\in Z \times C(Z;\bfb)\mid  x \in \overline{\nu_{b_1}\pi_1(p),\dots,\nu_{b_n}\pi_n(p)}\right\}
$$
where $\pi_i : \PP^{\bfa} \rightarrow \PP^{a_i}$ are the projection maps.
Algebraically, $\Gamma$ is defined by the vanishing of the maximal minors of the matrix whose columns are the 
$\nu_{b_{i}}(z_i)$ and $p$. Since the points $\nu_{b_{1}}(z_1), \dots, \nu_{b_{n}}(z_n)$
are linearly independent, the projection $\Gamma \to \PP^{\bfa}$ makes $\Gamma$ into a
$(n-1)$-plane bundle over $\P^{n}$. Restricting this bundle to 
a scheme $Z\subset\PP^{\bfa}$ we get a bundle $\Gamma_{Z} \to Z$, 
and the variety $C(Z;\bfb)$ is the image of $\Gamma_{Z}$ under the other projection.

When $n=2$ and $Z$ is a divisor in $\PP^{1}\times \PP^{1}$ of bidegree $(b'_{1}, b'_{2}) \leq (b_{1}, b_{2})$, 
the  correspondence scroll can also be realized as a variety of complexes (cf. \cite{DeSt}). 
We give two examples.

\begin{example}[Rational Normal Scrolls]\label{scroll from cplx}
  The rational normal scroll $\S(b_1,b_2)$ of dimension 2 is the determinantal variety in $\P^{b_1+b_2+1}$ defined by the vanishing of the $2$-minors of the $2\times (b_1+b_2)$ matrix
$$
\begin{pmatrix}
z_{1,0}&\dots&z_{1,b_1-1}&&z_{2,0}&\dots&z_{2,b_2-1}\\
z_{1,1}&\dots&z_{1,b_1}&&z_{2,0}&\dots&z_{2,b_2}
\end{pmatrix}
$$
where the $z_{i,j}$ are the homogeneous coordinates. But we can rewrite the ``mixed'' minors as products:
$$
\det \begin{pmatrix}
z_{1,i}&z_{2,j}\\
z_{1,i+1}&z_{2,j+1}
\end{pmatrix}
=
\begin{pmatrix}
z_{1,i}&z_{1,i+1}
\end{pmatrix}
*
\begin{pmatrix}
0&1\\
-1&0
\end{pmatrix}
*
\begin{pmatrix}
z_{2,j}\\
z_{2,j+1}
\end{pmatrix}.
$$
Thus the ideal of  $\S(b_1,b_2)$ may be written as the sum of three ideals: the ideal of $2$-minors of the 
matrix 
$$
M_1 :=
\begin{pmatrix}
 z_{1,0}&\dots&z_{1,b_1-1}\\
z_{1,1}&\dots&z_{1,b_1}
\end{pmatrix}\,,
$$
the ideal of $2$-minors of the matrix
$$
M_2:= \begin{pmatrix}
z_{2,0}&\dots&z_{2,b_2-1}\\
z_{2,0}&\dots&z_{2,b_2}
\end{pmatrix}\,,
$$
and the entries of the composition
$$
\begin{pmatrix}
z_{1,0}&z_{1,1}\\
\vdots&\vdots\\
z_{1,b_{1}-1}&z_{1,b_{1}}
\end{pmatrix}
*
\begin{pmatrix}
0&1\\
-1&0
\end{pmatrix}
*
\begin{pmatrix}
z_{2,0}&\dots&z_{2,b_2-1}\\
z_{2,1}&\dots&z_{2,b_2}
\end{pmatrix}\,.
$$

As we have noted, the scroll $\S(b_1, b_2)$ is the variety $C(\Delta; (b_1, b_2))$,
and we may think of the matrix 
$\begin{pmatrix}
0&1\\
-1&0.
\end{pmatrix}
$
as representing the coefficients of the defining equation 
$f = x_{1,0}x_{2,1}-x_{1,1}x_{2,0}$ of the diagonal $\Delta$. The reason this works is that modulo the ideals of of minors of the matrices $M_1$ and $M_2$ we may make the identifications
 $z_{i,j}\equiv x_{i,0}^{b_i-j}x_{i,1}^j$, so the elements of the composition are exactly the forms defining the ideal $J$
in Proposition~\ref{generators}.
\end{example}

\begin{example}[K3 Carpets \cite{EiSc}] The K3 carpet $X(b_1, b_2)$ is the correspondence scroll $C(2\Delta; \bfb)$ where $2\Delta$ denotes
the double of the diagonal in $\PP^1\times \P^1$. The equation of $2\Delta$ is 
$$
f^2 = (x_{1,0}x_{2,1}-x_{1,1}x_{2,0})^2 = (x_{1,0}x_{2,1})^2
-2(x_{1,0}x_{2,1})(x_{1,1}x_{2,0})+(x_{1,1}x_{2,0})^2.
$$
Thus, applying the reasoning and the notation of Example~\ref{scroll from cplx}, and noting the the coefficient of $f^2$ are $(1, -2,1)$, we see that the ideal of the 
K3 carpet is the sum of the ideal of minors of $M_1$, the minors of $M_2$, and the ideal of entries of the composition
$$
 \begin{pmatrix} 
 z_{1,0} &  z_{1,1} &  z_{1,2} \cr
 z_{1,1} &  z_{1,2} &  z_{1,3} \cr
\vdots & \vdots & \vdots\cr
 z_{1,b_1-2} &  z_{1,b_1-1} &  z_{1,b_1} \cr
 \end{pmatrix}
 *
 \begin{pmatrix} 
 0& 0 &  1\cr
 0 & -2 & 0 \cr
 1 & 0 & 0 \cr
 \end{pmatrix} 
  *
  \begin{pmatrix} 
 z_{2,0} &   z_{2,1} & \ldots & z_{2,b_2-2} \cr
 z_{2,1} &  z_{2,2} & \ldots &  z_{2,b_2-1} \cr
  z_{2,2}&  z_{2,3} & \ldots &  z_{2,b_2} \cr
 \end{pmatrix} \,.
 $$
\end{example}

\section{Nonsingularity}
It is interesting to ask when $C(Z;\bfb)$ is nonsingular. We may suppose that
$C(Z;\bfb)$ is irreducible so, leaving aside trivial cases, we may take $Z$ to
be an irreducible subscheme of $\P^{a_1}\times\cdots\times \P^{a_n}$. 
Moreover, if one of the 
$b_i$ is 0 then  $C(Z;\bfb)$ is a cone, so we assume that all $b_i\geq 1$.

The rational normal scrolls
$\Sigma(\bfb) = C(\Delta; \bfb)$, where $\Delta$ is the small diagonal in $\P^1\times\cdots \times\P^1$,  are nonsingular if all the $b_i$ are positive, but
in general the answer will depend on $\bfb$. 
For example $C(\P^1\times \P^1;\bfb)$ is the join of the rational normal curves of degrees $b_1, b_2$. 
If $b_1=b_2=1$, then this variety is $\P^3$, and is thus nonsingular. 
But if either $b_1$ or $b_2$ is greater than $1$, then the join becomes singular. 
In fact, the dimension is  3, as long as all the $b_i$ are positive; but
if $b_2>1$ then the linear span of $\nu_{b_2}(\P^1)$ has dimension greater than $ 1$, and one can see from this that the tangent  space at a point of the form  $\nu_{b_1}(p)$ will have dimension $b_2+2>3$.

Using a similar argument, we will characterize those $Z$ such that $C(Z;\bfb)$ is nonsingular for some $\bfb$ whose components $b_i$ are all greater than  $ 1$.

In the case $n=1$, the scheme $C(Z;b_1)\subseteq \P^N$ is the $b$-th Veronese embedding of $Z$, and is thus isomorphic to $Z\subset \P^{a_1}$. The following useful result is an analogue for $n>1$.

 It will be convenient to use a  basis-independent notation.
 To this purpose,  we write $\P^{a_i} = \P(V_i)$,
where $V_i$ is a $\Bbbk$-vector space of dimension $a_i+1$,
  so that $A = \Sym(\oplus_{i=1}^nV_i), S = \Sym( \oplus_{i=1}^n\Sym^{b_i}{V_i})$ and 
$$
 \P^N = \P\left(\oplus_{i=1}^n\Sym^{b_i}{V_i}\right) \supset \coprod_{i=1}^n \P\left(\Sym^{b_i}V_i\right).
$$

\begin{lemma}\label{multihomogeneous}
 Let $Z\subset \prod_{i=1}^n\P(V_i)$ be a subscheme and $\Lambda \subseteq \{1, \ldots, n\}$.
 The following three subschemes of  $\P(\oplus_{i\in \Lambda}\Sym^{b_i}V_i)\subseteq \P^N$ are equal:
\begin{itemize}
\item[$(i)$] the scheme $C\big(\pi_\Lambda(Z); \bfb_\Lambda\big)$, where $\pi_\Lambda: \prod_{i=1}^n\P(V_i)\to \prod_{i\in\Lambda}\P(V_i)$ denotes the natural projection and  $\bfb_\Lambda$ the subvector  $(b_i \,: \, i \in \Lambda)$;
\item[$(ii)$] the projection of $C(Z; \bfb) $ from the linear subspace $\P\big(\oplus_{i\notin \Lambda}\Sym^{b_i}V_i\big)\subseteq\P^N$;
\item[$(iii)$] the intersection $C(Z;\bfb)\cap \P\big(\oplus_{i\in \Lambda}\Sym^{b_i}V_i\big).$
\end{itemize}

\end{lemma}

\begin{proof} 
Denote for simplicity $U_i = \Sym^{b_i} V_i$.
Let $I\subseteq \Sym(\oplus_{i=1}^n V_i)$ be the saturated multigraded ideal of $Z$.
The ideal of $C\big(\pi_\Lambda(Z); \bfb_\Lambda\big) \subseteq \P(\oplus_{i\in \Lambda} U_i)$ is obtained by first intersecting $I$ with the subring $\Sym(\oplus_{i\in \Lambda} V_i)$ and then taking the preimage  in $\Sym(\oplus_{i\in \Lambda} U_i)$.
If we take the preimage of $I$ in $\Sym(\oplus_{i=1}^n U_i)$ first, and then intersect with the subring $\Sym(\oplus_{i\in \Lambda} U_i)$, we obtain the same ideal, 
thus the subschemes $(i)$ and $(ii)$ coincide.

Let $J\subseteq \Sym(\oplus_{i=1}^n U_i)$ be the saturated multigraded ideal of $C(Z;\bfb)$.
The  subscheme  $C(Z;\bfb)\cap \P(\oplus_{i\in \Lambda}U_i)$ of $\P(\oplus_{i\in \Lambda}U_i)$ is defined by the ideal 
$$
\frac{J+\big(\sum_{i\notin\Lambda} U_i\big)}{\big(\sum_{i\notin\Lambda} U_i\big)}.
$$ 
Since $J$ is multihomogeneous, the ideal $\big(\sum_{i\notin\Lambda} U_i\big)$ contains all the
generators of $J$ whose multidegrees have nonzero components outside $\Lambda$,
so the defining ideal of $C(Z;\bfb)\cap \P(\oplus_{i\in \Lambda}U_i)$ in $\P(\oplus_{i\in \Lambda}U_i)$ is generated by the classes of  elements of $J$ whose multidegrees
have nonzero components only in $\Lambda$.
On the other hand, the  ideal of 
the projection of $C(Z; \bfb) $ from  $\P(\oplus_{i\notin \Lambda}U_i)$ 
is $J\cap \Sym(\oplus_{i\in \Lambda} U_i)$,
and this has the same set of generators, whence the subschemes $(ii)$ and $(iii)$ coincide.
\end{proof}

\begin{thm}\label{TheoremNonSingular} Suppose that $Z\subset \P^{a_1}\times\cdots \times \P^{a_n}$ is a subscheme. The following conditions are equivalent:
\begin{enumerate}
 \item The correspondence scroll $C(Z;\bfb)$ is nonsingular for all $\bfb$.
 \item The correspondence scroll $C(Z;\bfb)$ is nonsingular for some $\bfb = (b_1,\dots, b_n)$
 with  $b_i\geq 2$ for all $i$.
 \item The scheme $Z$ is nonsingular and the projections $\pi_i: Z\to \P^{a_i}$ are 
 isomorphisms onto their images.
\end{enumerate}
\end{thm}

\begin{proof} 
Let $L\subset \P^{N}$ be the subspace $\P^{{a_1+b_1 \choose a_1}-1}= \P(\Sym^{b_{1}}(V_{1}))\subset \P^N$,
and let $\P^{N'} = \P(\oplus_{i=2}^{n}\Sym_{b_{i}}V_{i}\subset \P^{N})$ be the complementary subspace.

{\bf 3) $\Rightarrow$ 1)}
Suppose that $Z$ is nonsingular and that the projections $\pi_i: Z\to \P^{a_i}$ are 
 isomorphisms onto their images $Z_i$. 
 Since the components of $Z$ satisfy the same conditions as $Z$, we may assume that $Z$ is reduced and irreducible, and it follows that
 $C(Z;\bfb)$ is, too. 

As in Section~\ref{SectionBundles}, set
$$
\Gamma := \left\{(p, x)\in Z \times C(Z;\bfb)\,\big|\, x \in 
\overline{\nu_{b_1}\pi_1(p),\dots,\nu_{b_n}\pi_n(p)}\right\}.
$$
Each fiber of the projection to the first factor is isomorphic to $\P^{n-1}$. 
Since $Z$ is nonsingular, $\Gamma$ is nonsingular as well.
We claim that the projection of $\Gamma$ to the second factor is an isomorphism; 
that is, the fiber $F := F_{x}\subset Z$ of $\Gamma$ over any point $x\in C(Z;\bfb)$ is a reduced point. 

We  induct on the number of factors $n$.  
 If the point $x$ lies in $L$, 
then it lies in $\nu_{b_1}\pi_1(Z)$.
It follows that $F$ is contained in the fiber of the map
$\pi_i: Z\to \pi_i(Z)$, and this map is an isomorphism by hypothesis.
Hence, the fiber of $\Gamma$ over $x$ is a reduced point. In particular, this finishes the case $n=1$.

Now suppose $x$ does not lie  in  $L$, so that in particular $n>1$. Note that the fiber $F$ may be identified with the
fiber over $x$ of the projection 
$$
\Gamma\setminus (Z\times L)  \to C(Z;\bfb)\setminus L.
$$

By our hypothesis the natural projection of $Z$ to $\P^{N'}$ is an isomorphism onto its image $Z'$. 
Let $\bfb' = (b_2,\dots, b_n)$. 
By Lemma~\ref{multihomogeneous} we have  $C(Z';\bfb') = \pi'\big(C(Z;\bfb)\big)$ where $\pi'$ denotes the projection from $L$;
note that $\pi'$ is well-defined
away from $\nu_{b_1}\pi_1(Z) \subseteq L$.
Set
$$
\Gamma' := \left\{(p, x)\in Z' \times C(Z';\bfb')\,\big|\, x \in \overline{\nu_{b_2}\pi_2(p),\dots,\nu_{b_n}\pi_n(p)}\right\}.
$$
There is a commutative diagram of  maps
$$\begin{diagram}[small]
\Gamma \setminus (Z\times L) &\rTo&\Gamma'\\
\dTo&&\dTo\\
C(Z;\bfb)\setminus L &\rTo &C(Z';\bfb').
\end{diagram}
$$
 By induction, 
  the right hand vertical map is an isomorphism, 
so the fiber over $y=\pi'(x)$ is a reduced point $(q,y)$.
Thus
 $F\subset Z$ is contained in the fiber of the projection $Z \to Z'$ over $q$. By hypothesis, the projection $Z\to Z'$
 is an isomorphism, proving that $F$ is a reduced point.

{\bf 1) $\Rightarrow$ 2)} Trivial.

{\bf 2) $\Rightarrow$ 3)}
Suppose that $C(Z;\bfb)$ is nonsingular for some $\bfb$ with   $b_i\geq 2$ for all $i$. Let $Z_i = \pi_i Z$ be the projections. In the case $n=1$
we have $Z = Z_{1}\cong C(Z; \bfb)$, so we may assume $n>1$. 

We have $C(Z;\bfb) \cap L = \nu_{b_1} C(Z_1; b_{1})\cong Z_{1}$ by Lemma \ref{multihomogeneous}.
Fix a point $p\in Z$ and let $F \subseteq \P^{a_2}\times\cdots\times\P^{a_n}$ be the subscheme such that
$
\pi_1(p) \times F = \pi_1^{-1}(\pi_1(p))  \subseteq Z.
$
It suffices to prove that $Z_{1}$ is nonsingular and that $F$ is a reduced point.
Observe that  $\dim Z\leq \dim Z_1+\dim F$.

Let $\bfb' = (b_2, \ldots, b_n)$ and consider the subscheme $C(F;\bfb')\subseteq \P^{N'}$.
Note that $C(F;\bfb')$ is contained in $C(Z;\bfb)$. Moreover,
$C(Z;\bfb)$ contains the cone $\mathcal{T}_1$ over $C(F;\bfb')$ with vertex $\nu_{b_1}\pi_1(p)$. 
Note that $\dim \mathcal{T}_1 = \dim C(F;\bfb') + 1 = \dim F + n-1 $.

The tangent cone $\mathcal{T}$  to $C(Z;\bfb)$ at the point $\nu_{b_1}\pi_1(p)$ contains the tangent cone $\mathcal{T}_2$ to $\nu_{b_1} (Z_1)$ at $\nu_{b_1}\pi_1(p)$,
which satisfies $\dim \mathcal{T}_2 \geq\dim Z_1 $ and $\mathcal{T}_2\subseteq L$. 
On the other hand,  $\mathcal{T}$ obviously  contains the cone $\mathcal{T}_1$.

From the fact that $C(Z;\bfb)$ is nonsingular we deduce that 
$$
\dim \mathcal{T}=\dim C(Z;\bfb) = \dim Z +n-1\leq \dim Z_1+\dim F +n-1.
$$
Further, $\mathcal{T}$ is a linear space so $\mathcal{T}$  contains the linear spans $\mathcal L_{1}$ and $\mathcal L_{2}$
 of $\mathcal{T}_1$ and $\mathcal{T}_2$.
Since  $\mathcal{L}_1$ and $\mathcal{L}_2$ intersect only at the point $\nu_{b_1}\pi_1(p)$, we conclude that
\begin{align*}
\dim \mathcal{T}&\geq \dim \mathcal{L}_1 +\dim \mathcal{L}_2\\ 
&\geq \dim \mathcal{T}_1 +\dim \mathcal{T}_2\\
&=  \dim F + n-1 + \dim \mathcal{T}_2\\
&\geq \dim F + n-1 + \dim Z_{1}. 
\end{align*}
Thus all the inequalities are equalities, so   $\mathcal{T}_1, \mathcal{T}_2$ are  linear spaces, and
$\dim \mathcal T_{2} = \dim Z_{1}$. Thus $Z_1$ is nonsingular and 
 $\dim F = \dim Z - \dim Z_1$.

Since the cone $\mathcal{T}_1$ over $C(F;\bfb')$ is a linear space, it follows that $C(F; \bfb')$  is a linear
space  and thus $F$ is reduced. Further the projections of $C(F;\bfb')$ into each $\P^{{a_i+b_i \choose a_i}-1}$
 are linear spaces.
By Lemma \ref{multihomogeneous} these projections lie in  $\nu_{b_i}(\P^{a_i})$.
Since  $b_i\geq 2$, the Veronese $\nu_{b_i}(\P^{a_i})$ contains no linear spaces except reduced points.
Thus $C(F; \bfb') = n-2$. However, we also have $\dim C(F; \bfb') \dim F+n-2$, and it follows that $\dim F = 0$,
completing the argument.
\end{proof}

\section{Algebraic  properties}\label{SectionAlgebraicProperties}

In order to study further algebraic properties of correspondence scrolls, we introduce a  general multigraded operation.
To  each  vector $\bfb=(b_1, \ldots, b_n)\in \N_+^n$ we associate the finite index subgroup $H^{(\bfb)} := \langle b_1 \bfe_1, \ldots, b_n \bfe_n \rangle \subseteq  \mathbb{Z}^n$.
For a $\mathbb{Z}^n$-graded vector space $V$  we define 
$$
T_\bfb(V):= \bigoplus_{\bfd \in H^{(\bfb)}} V_\bfd.
$$
Notice that $T_\bfb(\cdot)$ is an exact functor on the category of $\mathbb{Z}^n$-graded vector spaces.

Recall the notation of the paper:
 $A$ is the Cox ring of $\P^\bfa = \P^{a_1}\times \cdots \times \P^{a_n}$, $Z$ a scheme defined by a multigraded ideal $I\subseteq A$,
and  $S$ is the homogeneous coordinate ring of $\P^N$.
Then the correspondence scroll is $C(Z;\bfb)= \Proj (T_\bfb(A/I))\subseteq \P^N$, 
where $T_\bfb(A/I)$ is regarded as a standard $\Z$-graded algebra.

\begin{remark}
For multigraded ideals $J $ of a $\Z^n$-graded algebra $R$,  the association 
$$
J \mapsto T_\bfb(J) = J\cap T_\bfb(R) \subseteq T_\bfb(R)
$$ preserves several  properties and operations of ideals, e.g.  prime, primary, radical, primary decomposition, intersection, sum, colon by  forms in $T_\bfb(R)$.
\end{remark}

If $R$ is a  Gorenstein ring, then its canonical module is $\omega_R \cong R(\bfv)$ for some vector $\bfv \in \Z^n$, known as the $a$-invariant of $R$ and denoted by $a(R)$.

\begin{thm}\label{TheoremGorenstein}
Let  $A$ be the Cox ring of $\P^\bfa = \P^{a_1}\times \cdots \times \P^{a_n}$, $I\subseteq A$  a multigraded ideal,
and  $R=A/I$. 
\begin{enumerate}
\item If $R$ is Cohen-Macaulay, then $T_\bfb(R)$ is Cohen-Macaulay;

\item if $R$ is a normal domain, then $T_\bfb(R)$ is a normal domain;

\item if $R$ is Gorenstein, then  $T_\bfb(R)$ is Gorenstein if and only if  $a(R) \in H^{(\bfb)}$;

\item if $R$ is Gorenstein with $a(R) = \bf0$, then  $T_\bfb(R)$ is Gorenstein with $a(R) = \bf0$.

\end{enumerate}
\end{thm}

\begin{proof}
For simplicity, denote $R' = T_\bfb(R)$. 
The ring extension  $R' \subseteq R$  is homogeneous with respect to the $\Z^n$ grading, and furthermore it is integral 
since for  each multigraded element $f\in R$ there exists a power lying in $R'$.
Denote their Krull dimension by   $d=\dim R = \dim R'$.

(1) and (2) follow from the fact that $R'$ is a direct summand of $R$ as $R'$-module,
see for instance \cite{HoEa}.

(3) 
We begin by showing that $T_\bfb(\cdot)$ behaves well with respect to local cohomology. 
In general, if $M$ is a graded $R$-module, then $T_\bfb(M)$ is a graded $R'$-module.
However, the functor  $T_\bfb(\cdot)$ does not preserve injectives, so we use the \v{C}ech complex.

If $f \in R$ is  multigraded with $\deg(f) \in H^{(\bfb)}$, then we have $T_\bfb(M_f) = (T_\bfb(M))_f$.
For the \v{C}ech complex of a graded sequence $\mathfrak{f}= f_1, \ldots, f_s$  with degrees in $H^{(\bfb)}$
 we obtain 
\begin{equation}\label{Cech}
 T_\bfb( C^\bullet (\mathfrak{f}; M)) =  C^\bullet (\mathfrak{f}; T_\bfb(M)).
\end{equation}

Let $J \subseteq R$ be a multigraded ideal,
then $J' = T_\bfb(J) $ is a multigraded ideal of $R'$. 
Since every multigraded  element  $f \in R$ has a power in $R'$,
 we may choose, up to radical, generators $\mathfrak{f}$ of $J$ with degrees in $H^{(\bfb)}$.  
From \eqref{Cech} and the fact that $T_\bfb(\cdot)$ is an exact functor, we conclude that local cohomology and $T_\bfb(\cdot)$ commute: for any $J,M,i$ we have
\begin{equation}\label{LocalCohomology}
 T_\bfb( H^i_J ( M)) =  H^i_{J'} (T_\bfb(M)).
\end{equation}

Since $R$ is Cohen-Macaulay, $R'$ is also Cohen-Macaulay by (1). Their graded canonical modules are 
 $$
 \omega_R = {^*\Hom_\Bbbk}(H^d_{R_{+}}(R), \Bbbk), \qquad \omega_{R'} = {^*\Hom_\Bbbk}(H^d_{{R'}_+}(R'), \Bbbk).
 $$
 where  $R_+, R'_+$ denote the respective homogeneous maximal ideals and  $^*\Hom_\Bbbk(\cdot, \Bbbk)$ is the Hom functor in the category of $\Z^n$-graded modules (see \cite{GW}).
This functor is exact and commutes with $T_\bfb(\cdot)$.
From \eqref{LocalCohomology} we conclude that $\omega_{R'}= T_\bfb(\omega_R)$.

Since  $R$ is a Gorenstein ring, then $\omega_R \cong R(\bfa)$  where $\bfa= a(R) \in \Z^n$ is the  $a$-invariant.
For a cyclic free module $R(\bfv)$,  we have that $T_\bfb(R(\bfv))$ is a free $R'$-module if and only if $\bfv \in H^{(\bfb)}$, in which case it is also cyclic.
We conclude that $R'$ is Gorenstein if and only if $a(R) \in H^{(\bfb)}$.

(4) If $R$ is Gorenstein with $a(R)=\bf0$ then we have the  graded isomorphism   $\omega_R\cong R$. 
Since  $\omega_{R'}= T_\bfb(\omega_R)$, we deduce that $\omega_{R'} \cong R'$, so  $R'$ is also Gorenstein with $a(R)=\bf0$.
\end{proof}

Notable examples include Gorenstein correspondence scrolls obtained from  complete intersections in $\P^{a_1}\times \cdots \times \P^{a_n}$, in particular from divisors.

\begin{cor}\label{CorollaryCompleteIntersection}
Let $Z \subseteq  \P^{a_1}\times \cdots \times \P^{a_n}$ be a complete intersection defined by multigraded forms  $f_1, \ldots, f_c$.
If there are $\lambda_i\in \Z$ such that
$$\deg(f_1) + \cdots + \deg(f_c) = (a_1 + 1+ \lambda_1 b_1,  \ldots, a_n + 1+ \lambda_n b_n)
$$
then  $C(Z;\bfb)$ is a Gorenstein projective scheme, and if $\lambda_i =0 $ for all $i $  then it is  Calabi-Yau.
\end{cor}

\begin{proof}
The $a$-invariant  of $\P^\bfa$ is $(a_1 + 1,  \ldots, a_n +1)$.
After going modulo the regular sequence $f_1, \ldots, f_c$ we obtain a complete intersection  with $a$-invariant equal to
$(a_1 + 1,  \ldots, a_n +1)-\deg(f_1) - \cdots - \deg(f_c)$.
Both statements follow directly  from the previous theorem.
\end{proof}

Next, using the work of Shibuta \cite{Sh}, we obtain information on   Gr\"obner bases of correspondence scrolls.
We say that a Gr\"obner basis is squarefree if its elements have squarefree leading monomials.

\begin{prop}\label{PropositionGrobner}
Let $I \subseteq A$ be a multigraded ideal,  $Z \subseteq \prod_{i=1}^n\A^{a_i+1}$  the scheme defined by $I$, 
and $\delta\geq 2 $  an integer.
If $I$ has a Gr\"obner basis  of forms of degree at most $\delta$, then $C(Z;\bfb)$  has a Gr\"obner basis  of forms of degree at most $\delta$ for every $\bfb$.
Moreover, if the Gr\"obner basis of $I$ is squarefree, then the Gr\"obner basis of  $C(Z;\bfb)$ will also be squarefree.
\end{prop}

\begin{proof}
Let  $\initial(I)\subseteq A$ be a monomial  initial ideal of $I$, with respect to a term order or integral weight. 
Let $H\subseteq S$ be the  ideal generated by all monomials whose image  via the map  $S\rightarrow A$ 
of Definition \ref{DefinitionCorrespondenceScroll} lies in   $\initial(I)$.
If $\initial(I)$ has generators of degree at most $\delta$, respectively, is a squarefree monomial ideal, then the same is true for $H$, cf. \cite[Lemma 2.6]{Sh}.

Since $T_\bfb(A)$ is the tensor product of the coordinate rings of $\nu_{b_i}(\P^{a_i})\subseteq \P^{{a_i +b_i\choose a_i}-1} $ over $\Bbbk$, 
a Gr\"obner basis for the kernel of  $ S \rightarrow A $ may be obtained as the union of Gr\"obner bases of each factor.
It follows  by \cite[Theorem 14.2]{Stu} that this kernel admits a  squarefree quadratic initial ideal $J\subseteq S$.

Finally,  there exists a term order on $S$ such that  $J+H $ is an initial ideal of $C(Z;\bfb)$, cf.
  \cite[Theorem 2.9]{Sh},
  and the desired statements follow.
\end{proof}

A special case of  Proposition \ref{PropositionGrobner} states that if $Z$ is defined by a Gr\"obner basis of quadrics, then $C(Z;\bfb)$ is also defined by a Gr\"obner basis of quadrics for all $\bfb$.
Blum \cite[Theorem 2.1]{Bl} proves a closely related statement:
he shows -- using the terminology of ``generalized Veronese subrings''-- that if $R$ is a Koszul algebra, then $T_\bfb(R)$ is also a Koszul algebra for all $\bfb$.

\section{Examples of correspondence scrolls}\label{SectionExamples}

In this section we list more examples of correspondence scrolls, in order to   show their ubiquity.
In some cases, our analysis offers an alternative point of view or a quicker proof of some results in the literature.

\begin{example}[Join variety]
Let  $X_1, \ldots, X_n$ be closed subschemes of general linear subspaces $\P^{a_1}, \ldots, \P^{a_n} $ in $ \P^N$.
If $Z = X_1 \times \cdots \times X_n$
then  $C(Z;\bf1)$ is the join variety of the $X_i$'s in $\P^N$.
Its degree is the product $\deg(X_1) \deg(X_2) \cdots \deg(X_n)$.
 \end{example}

The next two examples are  different generalizations of 2-dimensional rational normal scrolls.

\begin{example}[Determinantal ideals of square matrices]
Fix integers $2\leq r \leq n $.
Let $Z  \subseteq (\P^{n-1})^{n}$ be the subscheme defined by the $r$-minors of the  $n\times n$ generic matrix.
The correspondence scroll $C(Z;\bfb)$ is an irreducible  scheme in $\P^N$ of dimension $n^2-(n-r+1)^2-1$. 

By Theorem \ref{TheoremGorenstein} $C(Z;\bfb)$ is a Gorenstein scheme when $b_i $ divides $r-1$ for each $i$, 
since $Z$ is a Gorenstein scheme with $a$-invariant  $(1-r, \ldots, 1-r) \in \N^{n}$ (see \cite{BrHe}).
The $r$-minors form a squarefree  Gr\"obner basis, 
so by Proposition  \ref{PropositionGrobner}  $C(Z;\bfb)$ is  defined by a  squarefree Gr\"obner basis of multilinear forms for every $\bfb$.

When $n=r=2$ we obtain precisely the scrolls $\S(b_1,b_2)$.
As it is well-known,  this scroll is always Koszul, and the only such Gorenstein scroll is $\S(1,1)$;
however, for $r>2$ we have several examples of Gorenstein correspondence scrolls.
\end{example}

\begin{example}[Adjacent minors]
Fix integers $1 \leq m \leq n \in \N$. 
Let $Z\subseteq (\P^{m-1})^n$ be the subscheme defined by the  adjacent $m$-minors of the $m \times n$ generic matrix (see \cite{HoSu}).
These minors form a squarefree Gr\"obner basis and a regular sequence.

The scheme $C(Z;\bfb)\subseteq \P^N$  is  reduced   of dimension $mn-n+m-2$ .
By Proposition \ref{PropositionGrobner}, $C(Z;\bfb)$ is defined by a  sqaurefree Gr\"obner basis of  multilinear forms for every $\bfb$.
By Theorem \ref{CorollaryCompleteIntersection} the scheme $C(Z;\bfb)$  is Gorenstein if  $\bfb$ divides componentwise the vector
$(1,2, \ldots, m-1, m, m \ldots, m, m-1, \ldots, 2, 1) \in \N^n.$

In the case when $m=2$, the components of the subscheme $Z\subseteq (\P^{1})^n$ are described in  \cite{DES}. 
It follows that the number of components of the  correspondence scroll $C(Z;\bfb)$ is the Fibonacci number  $F_{n-1}$ (where $F_0=F_1=1$).
Each component is a join of  rational normal scrolls. 
In this case, $C(Z;\bfb)$ is  Koszul, and 
choosing $m=n=2$ gives rise to the scrolls $\S(b_1,b_2)$.
\end{example}

In the next examples we discuss generalizations of rational normal scrolls of all dimensions.
In preparation, we  compute the top Chern class of the small diagonal in the $n$-fold product of a projective space.

\begin{prop}\label{PropositionChern2Minors}
The 
Chern class of the small diagonal $\Delta \subseteq (\P^a)^n$  is
$$
c(\Delta) = \sum_{\alpha}\prod_{i=1}^n \zeta_i^{\alpha_i}
$$
where the sum ranges over all $\alpha\in \Z^n$ with $0 \leq \alpha_i \leq a$
and $\sum_i \alpha_i = a(n-1)$.
\end{prop}
\begin{proof}
In order to compute the coefficient of a monomial $\prod_{i=1}^n \zeta_i^{\alpha_i}$ with   $\sum_i \alpha_i =\codim (\Delta)= a(n-1)$,
we intersect $\Delta $ with a  product $L=\prod_{i=1}^n L_i$ where $L_i = \P^{\alpha_i} \subseteq \P^a$ is a general linear space.
However, this intersection is just one point, since
$$
 L \cap \Delta = \big\{ (p_1, \ldots, p_n) \, : p_i \in L_i \mbox{ and } p_1 = \cdots = p_n\big \} =
\big\{ (p, \ldots, p) \, : p \in \bigcap_{i=1}^n L_i\big\}
$$
and thus the coefficient is 1.
\end{proof}

\begin{example}[Diagonal]
Let $\Delta \subseteq (\P^a)^n$ be the small diagonal.
The correspondence scroll $C(\Delta;\bfb)$  is  an  irreducible  subscheme in $\P^N$ of dimension $a+n-1$, 
and it is nonsingular 
by Theorem \ref{TheoremNonSingular}.
By Theorem \ref{dim and deg} and Proposition \ref{PropositionChern2Minors},
the degree of $C(\Delta;\bfb)$ 
is   $\sum_{\alpha}\prod_{i=1}^n b_i^{\alpha_i}$
with sum ranging over all  $\alpha$
such that $\sum_i \alpha_i = \dim(\Delta) = a$.

Like the  rational normal scrolls,  $C(\Delta;\bfb)$ is arithmetically  Cohen-Macaulay and projectively normal by Theorem \ref{TheoremGorenstein}.
Since the $2$-minors of the $(a+1)\times n$ generic matrix form a squarefree Gr\"obner basis, $C(\Delta;\bfb)$ is defined by a squarefree Gr\"obner basis of quadrics by Proposition \ref{PropositionGrobner}.
The scrolls $\S(b_1,\ldots, b_n)$ are obtained when $a=1$.
\end{example}

\begin{example}[Closures of linear spaces]
A  generalization of the previous example, and hence of rational normal  scrolls, is obtained by considering an arbitrary  linear subspace 
$L \subseteq \A^m$ and its closure   $\widetilde{L} \subseteq \P^{a_1}\times \cdots \times \P^{a_n}$ where $\sum a_i = m$.
These are special  examples of Cartwright-Sturmfels ideals,  a large class of multigraded ideal with strong properties (see  \cite{CDG17}). 
For any $L$ and $\bfb$, the scheme $C\big(\widetilde{L};\bfb\big)\subseteq \P^N$  is  irreducible, arithmetically Cohen-Macaulay, and projectively normal by Theorem \ref{TheoremGorenstein} and \cite[Theorem 3.1]{CDG17}.
\end{example}

Non-trivial varieties of minimal degree, i.e.  rational normal scrolls and cones over the Veronese surface, are examples of correspondence scrolls.
One can also use the construction to produce some reducible schemes of minimal degree (see \cite{EGHP}).

\begin{example}[Small schemes]
Let $n_1, n_2 \in \N$ and consider the product $\P^\bfa =(\P^1)^{n_1} \times  (\P^1)^{n_2}$.
Let $\Delta_i \subseteq (\P^1)^{n_i}$ be the diagonal for $i =1,2$, and choose a point $p_i \in \Delta_i$.
Let $Z_1 = \Delta_1 \times \{p_2\}, Z_2 = \{p_1\} \times \Delta_2$ and consider $Z = Z_1 \cup Z_2 \subseteq \P^\bfa$.
For every $\bfb \in \N^{n_1+n_2}$, 
the correspondence scroll $C(Z;\bfb)$ is ``small scheme'' in the sense of \cite{EGHP}.
In fact, it is the union of the two linearly joined scrolls $C(Z_1;\bfb)$ and $C(Z_2;\bfb)$, that is
$$
C(Z_1;\bfb)\cap C(Z_2;\bfb) = \Span\left(C(Z_1;\bfb)\right)\cap \Span(C(Z_2;\bfb)) = C((p_1,p_2);\bfb).
$$
As a consequence, it has minimal degree and linear syzygies.
\end{example}

The construction of correspondence scrolls may be used to produce multiple structures on a given scheme with desirable properties.

\begin{example}[A rope on a line]
Let $d,n\in \N$. 
Consider the subscheme $Z = \Delta \cap X \subseteq (\P^1)^n$, where $\Delta$ is the diagonal and $X$ is the $d$-th thickening of a point $p \in \Delta$.
The scheme $C(Z;\bfb)$  is a multiple structure of degree $d$ on a line in $\P^N$ sitting inside the rational normal scroll $C(\Delta; \bfb)$.
For $d=2$, such structure is called a ``rope'' (see \cite[Remark 2.10]{NNS}) and is defined by a Gr\"obner basis of quadrics.
\end{example}

\begin{example}[Double structure on Veronese varieties]
Let $d,n \in \N$. 
Let $V$ be the image of the  embedding 
$\P^n \xhookrightarrow{\nu_d} \P^{{n + d \choose n}-1}  \subseteq  \P^{{n + d \choose n}+{n + d-1 \choose n}-1}$.
Let $Z\subseteq \A^{n+1}\times\A^{n+1}$ be  the subscheme  defined by the $2$-minors of the generic $(n+1)\times 2$ matrix and by all the $d$-forms in the variables $x_{1,0},\ldots, x_{1,n}$.
The correspondence scroll $C(Z;\bfb)$ where $ \bfb = (d,d-1)$ is a double structure on the Veronese variety $V$.

When $n=1$, $V$ is  a rational normal curve of degree $d$ in $\P^{2d}$.  
The curve $C(Z;\bfb)$ we obtain is a double structure on $V$ with linear syzygies and degree $2d$ (i.e. the same resolution as the rational normal curve in $\P^{2d}$),
cf. the main result of  \cite{Ma}.
When $n=2, d=2$ we obtain a Gorenstein double structure $X=C(Z;\bfb)$ on a Veronese surface $V \subseteq \P^5 \subseteq \P^8$.
In fact, $X$ is defined by the $2$-minors of a $4\times 4$ symmetric matrix obtained from a generic symmetric matrix by setting the last entry equal to 0.
Observe that in both special cases there exists no such double structure inside the linear span of $V$.
\end{example}

\begin{example}[Canonically embedded balanced ribbons]
In \cite{DFS} the authors study  canonically embedded balanced ribbons of odd genus $g$, motivated by their role  in \cite{AFS}.
For each $g=2b+1$, with $b\geq 1$, we can realize this embedded curve as the  section of  
  $C(2\Delta;\bfb)\subseteq \P^{2b+1}$ by the hyperplane  $x_{0,b}=x_{1,0}$, where $ \bfb = (b,b)$ and $\Delta \subseteq \P^1\times \P^1$ is the diagonal. 
Since the carpet is Calabi-Yau and arithmetically Cohen-Macaulay by Corollary \ref{CorollaryCompleteIntersection}, the hyperplane section is a canonical curve. 
This ribbon satisfies the analogue of Green's Conjecture, 
as proven in \cite{D}. 
\end{example}

\begin{example}[Reducible K3 surfaces]\label{K3 example}
As explained in~\cite{EiSc}, one can make a degenerate K3 surface as the union of two scroll surfaces that have the same type, $\Sigma(b_1,b_2)\subseteq\P^{b_1+b_2+1}$, and
meet along the rational normal curves $\Sigma(b_1)$ and $\Sigma(b_2)$ as well as two rulings---a reducible curve of arithmetic genus 1.

We have constructed the rational normal scroll $\Sigma(b_1,b_2)$ as the 
correspondence scroll associated to the diagonal $\Delta\subset \P^1\times \P^1$, which 
we may think of as the graph of the identity map. Let 
$\Delta'$ be the graph of the non-identity automorphism $\sigma_t$ of $\P^1$ that is defined, in a suitable coordinate system, by multiplication by a scalar $t\in \Bbbk\setminus\{0,1\}$. 
The examples
from~\cite{EiSc} can then be constructed as
$X_{\sigma_t}(\bfb) : = C(\Delta\cup\Delta'; \bfb).$ Since $\Delta\cup\Delta'$ is a divisor of type $(2,2)$ in $\P^1\times \P^1$, Corollary~\ref{CorollaryCompleteIntersection} gives a second proof that $X_{\sigma_t}$ is Calabi-Yau and Cohen-Macaulay,
thus $X_{\sigma_t}(\bfb)$ is  a reducible K3  surface

Note that there is another type of automorphism $\sigma'_t$ of $\P^1$: those that correspond to addition by
a  scalar $t\in \Bbbk \setminus\{0\}$. The automorphism $\sigma_t$ has two fixed points (in the given coordinate system they are 0 and $\infty$); the automorphism $\sigma'_t$ by contrast has only one (namely, $\infty$).
The paper \cite{EiSc}  treats only the case
where $\sigma$ has two distinct fixed points,
 but Corollary~\ref{CorollaryCompleteIntersection}  applies equally to both parts.

The ideal $J$ of the intersection $E:= \Delta\cap\Delta'\subseteq \A^4$ is the complete intersection of two bilinear forms $\ell, \ell'$.
Since $J$ has codimension 2 and is contained in both irrelevant ideals $(x_{0,0}, x_{0,1})$
and $(x_{1,0}, x_{1,1})$, these must be
associated primes. Both $\Delta$ and $\Delta'$ have class $\z_1+\z_2$ in the Chow ring, so the ``relevant'' part of $\Delta\cap\Delta'$ has class $2\z_1\z_2$. This is the class of two points of $\P^1\times\P^1$, that is, of two 2-planes through the origin in $\A^4$ other than the ``irrelevant ideals". Since the degree of the intersection is 4, and a complete intersection is unmixed,
We must have 
$$
J = (x_{0,0}, x_{0,1}) \cap (x_{1,0}, x_{1,1})\cap J'
$$
where $J$ is either primary or the the intersection of 2 distinct primes.

Returning to the construction we see that the two possibilities correspond to the two types of automorphism $\sigma$: if $\sigma$ has two distinct fixed points, then $C(E; \bfb)$ consists of 4 reduced curves: the two rational normal curves 
$C(V(x_{0,0}, x_{0,1}); \bfb)$,
$C(V(x_{0,0}, x_{0,1}); \bfb)$,
and the two distinct or one double line from the ruling, corresponding to 
$C(V(J'); \bfb)$. 
\end{example}

\subsection*{Acknowledgments}
The first author was partially supported by NSF grant No. 1502190. 
He would like to thank Frank-Olaf Schreyer, who pointed out in their joint work that the K3 carpets could be regarded as coming from correspondences.
The second author was supported by  NSF grant No. 1440140 
while he was a Postdoctoral Fellow at the Mathematical Sciences Research Institute in Berkeley, CA. 
He would like to thank Aldo Conca and Matteo Varbaro for some helpful comments.

\end{document}